\documentclass[11pt]{amsart}

\usepackage[T1]{fontenc}
\usepackage{lmodern}
\usepackage{microtype}
\usepackage[margin=1.34in]{geometry}

\usepackage{amsmath,amssymb,amsfonts,mathtools}
\usepackage{amscd}
\usepackage{latexsym}
\usepackage{mathrsfs}
\usepackage{accents}
\usepackage{esint}
\usepackage{braket}

\usepackage{graphicx}
\usepackage{enumitem}

\usepackage{xcolor}
\usepackage{aliascnt}
\usepackage[
  colorlinks=true,
  linkcolor=blue,
  citecolor=blue,
  urlcolor=blue
]{hyperref}
\usepackage[nameinlink,noabbrev]{cleveref}

\hypersetup{
  pdftitle={No compromise in the liquid drop model},
  pdfauthor={Otis Chodosh and Matilde Gianocca}
}

\setlist{itemsep=3pt}
\allowdisplaybreaks
\numberwithin{equation}{section}

\theoremstyle{plain}

\newtheorem{theorem}{Theorem}[section]

\newaliascnt{proposition}{theorem}
\newtheorem{proposition}[proposition]{Proposition}
\aliascntresetthe{proposition}

\newaliascnt{lemma}{theorem}
\newtheorem{lemma}[lemma]{Lemma}
\aliascntresetthe{lemma}

\newaliascnt{corollary}{theorem}
\newtheorem{corollary}[corollary]{Corollary}
\aliascntresetthe{corollary}

\theoremstyle{remark}

\newaliascnt{remark}{theorem}
\newtheorem{remark}[remark]{Remark}
\aliascntresetthe{remark}

\crefname{theorem}{theorem}{theorems}
\Crefname{theorem}{Theorem}{Theorems}

\crefname{proposition}{proposition}{propositions}
\Crefname{proposition}{Proposition}{Propositions}

\crefname{lemma}{lemma}{lemmas}
\Crefname{lemma}{Lemma}{Lemmas}

\crefname{corollary}{corollary}{corollaries}
\Crefname{corollary}{Corollary}{Corollaries}

\crefname{remark}{remark}{remarks}
\Crefname{remark}{Remark}{Remarks}
\theoremstyle{plain}
\newtheorem{theo}[theorem]{Theorem}

\theoremstyle{definition}

\theoremstyle{remark}

\crefalias{theo}{theorem}
\crefalias{prop}{proposition}
\crefalias{lemm}{lemma}
\crefalias{coro}{corollary}
\crefalias{rema}{remark}

\newcommand{\RR}{\mathbb{R}}

\newcommand{\R}{\mathbb{R}}

\newcommand{\cE}{\mathcal{E}}

\DeclareMathOperator{\Div}{div}

\DeclareMathOperator{\capacity}{Cap}

\let\oldmarginpar\marginpar
\renewcommand{\marginpar}[1]{%
  \oldmarginpar[\raggedleft\footnotesize #1]{\raggedright\footnotesize #1}}

\title[No compromise in the liquid drop model]{No compromise in the liquid drop model}

\author[O. Chodosh]{Otis Chodosh}
\address{Department of Mathematics, Stanford University, Stanford, CA 94305, USA}
\email{ochodosh@stanford.edu}

\author[M. Gianocca]{Matilde Gianocca}
\address{Department of Mathematics, ETH Z\"urich, R\"amistrasse 101, 8092 Z\"urich, Switzerland}
\email{matilde.gianocca@math.ethz.ch}

\date{}

\begin{document}

\begin{abstract}
We characterize minimizers in Gamow's liquid drop problem. 
\end{abstract}

\maketitle

\section*{AI usage statement}
This paper describes results that were obtained by ChatGPT 5.6 Pro over a series of chats, without significant assistance from the authors. We have checked and reworked the proof but the fundamental strategy remains close to the original output. This article does not contain AI-written text. 

\section{Introduction}

For $\Omega\subset \RR^3$ measurable we define Gamow's liquid drop functional
\begin{equation}\label{eq:E}
\cE(\Omega) : = P(\Omega) + D(\Omega), 
\end{equation}
where
\begin{equation}\label{eq:D}
D(\Omega) : = \frac 12 \iint_{\Omega\times\Omega} \frac{dxdy}{|x-y|}
\end{equation}
is the Coulomb energy and 
\begin{equation}\label{eq:P}
P(\Omega) = \sup\left\{ \int_\Omega \Div X : X \in C^1_c(\RR^3;\RR^3), |X|\leq 1\right\}
\end{equation}
 is the perimeter of $\Omega$. Gamow introduced the functional \eqref{eq:E} in 1928 to model the nucleus of an atom \cite{Gamow}. See \cite{CMT:Notices} for discussions concerning the mathematical aspects of this problem.
 
We are interested in the problem of minimizing $\cE(\cdot)$ for sets of a fixed volume. Set
\begin{equation}\label{eq:mass-to-split}
V_* : = \frac{2-2^{\frac 23}}{2^{\frac23}-1} \frac{|B_1| P( B_1)}{\frac 12 \iint_{B_1\times B_1}|x-y|^{-1} dxdy} = 5\frac{2-2^{\frac 23}}{2^{\frac23}-1} \approx 3.51.
\end{equation}
The relevance of \(V_*\) is that two balls of volume \(V/2\) placed very far
apart have lower energy than a single ball of volume \(V\), whenever
\(V>V_*\). The phrase ``no compromise'' refers to the fact that, at fixed
volume, the ball minimizes perimeter but maximizes Coulomb energy.

The main result is as follows.
\begin{theo}\label{theo:main}
For \(V\le V_*\), round balls of volume \(V\) uniquely minimize
\(\cE(\cdot)\) among sets of fixed volume \(V\). For \(V>V_*\), no minimizer
exists among sets of fixed volume \(V\).
\end{theo}

This conjecture has appeared in several places, including
\cite{ChoksiPeletier,FrankLieb:compact-least-density,
FrankNam:existence-nonexistence}.

Theorem \ref{theo:main} resolves the question (cf.\ \cite{FrankLieb:compact-least-density}) of the minimal binding energy. 
\begin{corollary}\label{cor:binding}
The minimal binding energy satisfies
\[
e_*:=\inf_{0<|\Omega|<\infty}\frac{\cE(\Omega)}{|\Omega|}
=3\left(\frac{9\pi}{5}\right)^{1/3}.
\]
The infimum is attained by
balls of volume \(5/2\).
\end{corollary}
See \cite{KnupferMuratovNovaga, FLS:low-density} for some related considerations.

\subsection{Previous work} The first-named author and Ruohoniemi proved \cite{CR} that the balls minimize for volumes $V\leq 1$. Earlier work of Kn\"upfer--Muratov showed \cite{KnupferMuratov1,KnupferMuratov2} (cf.\ \cite{BonaciniCristoferi,Julin,MuratovZaleski,FigalliFuscoMaggiMillot,ChoksiNeumayerTopaloglu}) that balls uniquely minimize $\cE(\cdot)$ for very small $V$.  Kn\"upfer--Muratov \cite{KnupferMuratov2} and Lu--Otto \cite{LuOtto} proved non-existence for sufficiently large volumes. The quantitative estimate of non-existence for $V > 8$ was obtained by Frank--Killip--Nam  \cite{FrankKillipNam} and very recently improved to nonexistence for $V\geq 7.5$ by Schulz \cite{Schulz}. Finally, we note that Frank--Nam proved \cite{FrankNam:existence-nonexistence} existence of a minimizer in the full range $V \leq V_*$. See also \cite{AlbertiChoksiOtto,ChoksiPeletier1,ChoksiPeletier,RenWei:torus,CicaleseSpadaro,RenWei:2tori,KnupferMuratovNovaga:low-density,Frank:non-spherical-drops,EmmertFrankKonig}. 

\subsection{Description of the strategy} The strategy (as in the first-named author's work with Ruohoniemi \cite{CR}) is to integrate the first-variation formula 
\[
H + v_\Omega = \lambda 
\]
where $v_\Omega(x) = \int_{\Omega}\frac{dy}{|y-x|}$ satisfied by a minimizer $\Omega$ along $\partial\Omega$ (more precisely, one should fill any bounded regions in $\RR^3\setminus\Omega$ to obtain $K\supset\Omega$ but we ignore this distinction here).  The key idea used here is to introduce the capacitary potential $u$ satisfying $\Delta u =0$ on the complement of $\Omega$, with $u=1$ on $\partial\Omega$ and $u\to0$ at infinity. This is a natural quantity to consider due to its close relationship with the Coulombic energy term $D(\Omega)$ (cf.\ \Cref{lem:equilibrium}). 

In contrast with \cite{CR} where the first variation was integrated over $\partial\Omega$ directly, here it is weighted by $|\nabla u|$. 
\[
\int_{\partial\Omega} H|\nabla u| + \int_{\partial\Omega} v_\Omega |\nabla u| = \lambda \int_{\partial\Omega} |\nabla u| = 4\pi \lambda \capacity(\Omega).
\]
the second term is seen to be equal to $4\pi |\Omega|$ after two integration by parts (cf.\ \Cref{lem:equilibrium}).

At this point we use an estimate (cf.\ \Cref{prop:capacity})
\[
\int_{\partial\Omega} H |\nabla u| \geq 4 \int_{\partial\Omega} |\nabla u|^2 -8\pi \geq 8\pi.
\]
This generalizes an estimate obtained by Agostiniani--Mazzieri \cite{AM} (the improvement comes from using Gauss--Bonnet on a specific term; this only applies for this specific weight and only in $\RR^3$). This gives a strong inequality relating the perimeter, Coulomb energy, volume, and capacity of a minimizer. Combined with explicit comparisons (against one or two balls, in the appropriate regime) this suffices to prove \Cref{theo:main}. 

As an aside, we observe that the monotone quantity used in the proof of \Cref{prop:capacity} is the same as the one used by Israel in his proof of static uniqueness of Schwarzschild \cite{Isr67} (cf.\ \cite{Rob77}). Related capacitary methods have been used to prove Minkowski inequalities \cite{FMP19,AFM22,BFM24} and other geometric inequalities \cite{AMO22,AMMO25}.

\subsection*{Acknowledgements}
O.C. was partially supported by a Terman Fellowship and an NSF grant
(DMS-2304432). M.G. thanks the Department of Mathematics at Stanford University
for its hospitality during the completion of this work and M. Badran for
introducing her to the problem.

\section{The capacitary estimate}\label{sec:cap}

Let $\Omega\subset\R^3$ be a bounded connected $C^3$ stationary domain of
volume $V$. Then
\[
H+v_\Omega=\lambda\quad\text{on }\partial\Omega,
\qquad
v_\Omega(x):=\int_\Omega\frac{dy}{|x-y|},
\]
where $H$ is the sum of the principal curvatures for the outward normal.
All hypersurface integrals are taken with respect to the induced area measure.

Let \(U_\infty\) be the unbounded component of
\(\R^3\setminus\overline\Omega\), and set
\[
K:=\R^3\setminus U_\infty.
\]
Then \(K\) has \(C^3\) boundary, \(K\) and \(\R^3\setminus K\) are connected,
and
\[
\overline\Omega\subset K,
\qquad
\partial K\subset\partial\Omega,
\qquad
P(K)\le P(\Omega).
\]
The outward normals agree on \(\partial K\). Let \(u\) be the capacitary
potential of \(K\):
\[
\Delta u=0\quad\text{in }\R^3\setminus K,
\qquad
u=1\quad\text{on }\partial K,
\qquad
u\to0\quad\text{at infinity}.
\]
We use the normalization
\[
\capacity(K)
:=
\frac1{4\pi}\int_{\R^3\setminus K}|\nabla u|^2
=
\frac1{4\pi}\int_{\partial K}|\nabla u|.
\]
Since $\Delta u=0$,
\[
\int\limits_{\{u=s\}}|\nabla u|=4\pi\capacity(K)
\]
for every regular $s\in(0,1)$. For $t\ge1$, set
\[
\Phi(t):=t^2\int\limits_{\{u=1/t\}}|\nabla u|^2.
\]
For a ball, $\Phi\equiv4\pi$. In general,
\cite[Theorem~1.1(ii)--(iii)]{AM} shows that $\Phi\in C^1([1,\infty))$ is
nonincreasing and convex, and \cite[(1.6)]{AM} gives
\[
\lim_{t\to\infty}\Phi(t)=4\pi.
\]

For \(t\ge1\), write
\[
\Sigma_t:=\{u=1/t\}.
\]
Every regular \(\Sigma_t\) is connected by \Cref{lem:levels}, and \(t=1\)
is regular by the Hopf lemma. All geometric quantities on \(\Sigma_t\) are
computed with respect to \(-\nabla u/|\nabla u|\). 
\begin{lemma}\label{lem:boundary-estimate}
For every regular \(t\ge1\),
\[
\int_{\Sigma_t}H|\nabla u|
\ge
4t\int_{\Sigma_t}|\nabla u|^2-\frac{8\pi}{t}.
\]
\end{lemma}

\begin{proof}
By \Cref{lem:bochner},
\[
\Delta|\nabla u|
\ge
|\nabla u|
\left(
|\mathrm{II}|^2+
|\nabla^\top\log|\nabla u||^2
\right)
\]
in the sense of distributions. Set
\[
X:=\nabla|\nabla u|-4u^{-1}|\nabla u|\nabla u.
\]
Using
\[
\langle\nabla|\nabla u|,\nabla u\rangle
=
H|\nabla u|^2,
\qquad
|\mathrm{II}|^2=H^2-2\kappa,
\]
where \(\kappa\) is the Gauss curvature of the level set, we obtain
\[
\operatorname{div}X
\ge
|\nabla u|
\left[
-2\kappa
+
|\nabla^\top\log|\nabla u||^2
+
\bigl(H-2u^{-1}|\nabla u|\bigr)^2
\right]
\]
in the sense of distributions.

Put \(s=t^{-1}\), and let \(0<\varepsilon<s\) be a regular value. On
\(\{u=r\}\), the outward flux of \(X\) from \(\{u<r\}\) is
\[
H|\nabla u|-4r^{-1}|\nabla u|^2.
\]
Testing the distributional inequality above with smooth approximations of
\(\mathbf 1_{\{\varepsilon<u<s\}}\), and then using the coarea formula, gives
\[
\begin{aligned}
&\int_{\{u=s\}}
\left(H|\nabla u|-4s^{-1}|\nabla u|^2\right)
-
\int_{\{u=\varepsilon\}}
\left(H|\nabla u|-4\varepsilon^{-1}|\nabla u|^2\right)\\
&\qquad\ge
\int_\varepsilon^s
\left[
\int_{\{u=r\}}
\left(
|\nabla^\top\log|\nabla u||^2
+
\bigl(H-2r^{-1}|\nabla u|\bigr)^2
\right)
-
2\int_{\{u=r\}}\kappa
\right]dr.
\end{aligned}
\]
The expansion
\[
u(x)=\frac{\capacity(K)}{|x|}+O(|x|^{-2})
\]
which may be differentiated twice (cf.\ \cite[(1.5)]{AM}) implies that 
\[
\int_{\{u=\varepsilon\}}
\left(H|\nabla u|-4\varepsilon^{-1}|\nabla u|^2\right)
=O(\varepsilon).
\]
For almost every \(r\), the level \(\{u=r\}\) is regular and connected, so
Gauss--Bonnet gives
\[
-2\int_{\{u=r\}}\kappa
=
-8\pi+8\pi\,\operatorname{genus}(\{u=r\})
\ge-8\pi.
\]
Letting \(\varepsilon\downarrow0\), we obtain
\[
\int_{\{u=s\}}
\left(H|\nabla u|-4s^{-1}|\nabla u|^2\right)
\ge-8\pi s.
\]
Since \(s=t^{-1}\), this proves the claim.
\end{proof}

\begin{remark}\label{rem:second-variation}
The estimate in \Cref{lem:boundary-estimate} also follows from the variation
formulas of \cite{AM}. Set
\[
Q(t):=-\Phi'(t)-\frac2t\bigl(\Phi(t)-4\pi\bigr).
\]
By \cite[Remark~1 and (1.10)--(1.11)]{AM}, Gauss--Bonnet, and the
connectedness of \(\Sigma_t\),
\[
Q'(t)
=-\frac1{t^2}
\left[
\int_{\Sigma_t}
\left(
\bigl(H-2t|\nabla u|\bigr)^2
+|\nabla^\top\log|\nabla u||^2
\right)
+8\pi\,\operatorname{genus}(\Sigma_t)
\right]
\le0
\]
at almost every regular parameter. Since the regular parameters have full
measure, convexity of \(\Phi\) gives, for \(1\le a<b\),
\[
\left[
-\frac2t\bigl(\Phi(t)-4\pi\bigr)
\right]_{t=a}^{t=b}
\le
\int_a^b\Phi''(t)\,dt
\le
\Phi'(b)-\Phi'(a).
\]
Thus \(Q(b)\le Q(a)\). Since \(\Phi(t)\to4\pi\) and \(\Phi'(t)\to0\), we
have \(Q(t)\to0\). Hence \(Q\ge0\), and the first-variation formula
recovers \Cref{lem:boundary-estimate}.
\end{remark}

\begin{lemma}\label{lem:equilibrium}
For every measurable \(A\subset K\),
\[
\frac1{4\pi}\int_{\partial K}v_A|\nabla u|=|A|,
\qquad
v_A(x):=\int_A\frac{dy}{|x-y|}.
\]
\end{lemma}
\begin{proof}
We recall that $\Delta v_A = -4\pi \mathbf{1}_A$. We extend $u$ to be constant $u=1$ in $K$. We may integrate by parts first into $\RR^3\setminus K$ and then into $K$ to find
\[
\int_{\partial K}v_A|\nabla u|= - \int_{\partial K}\nabla_\nu v_A = -\int_{K} \Delta v_A = 4\pi |A| 
\]
This completes the proof. 
\end{proof}
\begin{proposition}\label{prop:capacity}
Let $\Omega\subset\R^3$ be a bounded connected $C^3$ stationary domain of
volume $V$, and let $u$ be the capacitary potential of its hull $K$. Then
\[
\begin{aligned}
\lambda\capacity(K)-V
&=\frac1{4\pi}\int_{\partial K}H|\nabla u|
\ge\frac1\pi\int_{\partial K}|\nabla u|^2-2\\
&\ge\max\left\{
2,\,
\frac{16\pi\capacity(K)^2}{P(\Omega)}-2
\right\}.
\end{aligned}
\]
Consequently,
\begin{equation}\label{eq:stationary-lower}
\lambda\sqrt{\frac{P(\Omega)}{4\pi}}
\ge
\begin{cases}
V+2,&0<V\le6,\\[1mm]
4\sqrt{V-2},&V\ge6.
\end{cases}
\end{equation}
\end{proposition}

\begin{proof}
Applying \Cref{lem:equilibrium} with $A=\Omega$ and integrating the
Euler--Lagrange equation against $(4\pi)^{-1}|\nabla u|$ gives the equality
in the first line. The first inequality follows from
\Cref{lem:boundary-estimate} with $t=1$, while monotonicity of $\Phi$ and
Cauchy--Schwarz give the second:
\[
\int_{\partial K}|\nabla u|^2
=
\Phi(1)\ge4\pi,
\]
and
\[
\int_{\partial K}|\nabla u|^2
\ge
\frac{\left(\int_{\partial K}|\nabla u|\right)^2}{P(K)}
\ge
\frac{16\pi^2\capacity(K)^2}{P(\Omega)}.
\]

Set
\[
s:=\capacity(K)\sqrt{\frac{4\pi}{P(\Omega)}}.
\]
Then
\[
\lambda\sqrt{\frac{P(\Omega)}{4\pi}}
\ge
\max\left\{
\frac{V+2}{s},\,
4s+\frac{V-2}{s}
\right\}.
\]
If $s\le1$ the estimate follows from considering only the first term. On $s\ge1$, the second has minimum $V+2$
for $V\le6$ and $4\sqrt{V-2}$ for $V\ge6$.
\end{proof}

\section{Proofs of the theorems}

By \cite[Theorem~2.7]{BonaciniCristoferi}, every minimizer has a bounded
representative with $C^{3}$ boundary. It is connected, since translating one
component to infinity strictly decreases the Coulomb interaction. It is
stationary, so using a scaling argument in the first variation of energy (cf. \cite[Lemma 13]{CR}) gives
\begin{equation}\label{eq:dilation}
3V\lambda
=2P(\Omega)+5D(\Omega)
=5\cE(\Omega)-3P(\Omega).
\end{equation}
The proofs of the theorems now follow by comparing
\eqref{eq:stationary-lower} with the one- and two-ball competitors. Let
\(B_V\) be a ball of volume \(V=\frac 43 \pi R^3\), and set
\[
P_B:=P(B_V)=4\pi R^2=\frac{3V}{R},
\qquad
D(B_V)=\frac V5P_B.
\]
See \cite[Lemma 2]{CR} for a calculation of $D(B_V)$. 
We then define 
\[
\delta:=\sqrt{\frac{P(\Omega)}{P_B}}-1\ge0
\]
by the isoperimetric inequality. Recall that $\delta=0$ implies that $\Omega$ is a ball.
\begin{proof}[Proof of \Cref{theo:main}, balls minimize up to \(V_*\)]
By \cite[Theorem~1]{FrankNam:existence-nonexistence}, a minimizer
\(\Omega\) exists. Comparison with \(B_V\) gives
\[
\cE(\Omega)\le\cE(B_V)=\frac{V+5}{5}P_B,
\qquad
P(\Omega)=P_B(1+\delta)^2.
\]
Hence, by \eqref{eq:dilation},
\[
\begin{aligned}
\lambda\sqrt{\frac{P(\Omega)}{4\pi}}
&=
\frac{R(1+\delta)}{3V}
\bigl(5\cE(\Omega)-3P(\Omega)\bigr)\\
&\le
\bigl(V+5-3(1+\delta)^2\bigr)(1+\delta).
\end{aligned}
\]
Since \(V\le V_*<4\), \eqref{eq:stationary-lower} gives the opposite bound
\[
\lambda\sqrt{\frac{P(\Omega)}{4\pi}}\ge V+2.
\]
If \(\delta>0\), then
\[
V+2-\bigl(V+5-3(1+\delta)^2\bigr)(1+\delta)
=
\delta\bigl(4-V+9\delta+3\delta^2\bigr)>0,
\]
a contradiction. Thus \(\delta=0\), and the equality case of the
isoperimetric inequality shows that \(\Omega\) is a translate of \(B_V\).
\end{proof}

\begin{proof}[Proof of \Cref{theo:main}, no minimizer after \(V_*\)]
By \cite{FrankKillipNam}, it remains to consider \(V_*<V\le8\). Suppose
that a minimizer \(\Omega\) exists, and set
\[
L:=2^{-2/3}(V+10).
\]
Comparison with two equal balls whose distance tends to infinity gives
\[
\cE(\Omega)
\le
2\cE(B_{V/2})
=
\left(
2^{1/3}+\frac{V}{5\,2^{2/3}}
\right)P_B
=
\frac L5P_B.
\]
Consequently, \eqref{eq:dilation} gives
\[
\begin{aligned}
\lambda\sqrt{\frac{P(\Omega)}{4\pi}}
&=
\frac{R(1+\delta)}{3V}
\bigl(5\cE(\Omega)-3P(\Omega)\bigr)\\
&\le
\bigl(L-3(1+\delta)^2\bigr)(1+\delta).
\end{aligned}
\]
The same comparison gives
\begin{equation}\label{eq:two-ball-perimeter}
(1+\delta)^2
=
\frac{P(\Omega)}{P_B}
\le
\frac{\cE(\Omega)}{P_B}
\le
\frac L5
<
2^{4/3}.
\end{equation}

If \(V_*<V\le6\), then \eqref{eq:two-ball-perimeter} and
\(2^{-2/3}(V_*+10)=V_*+5\) give
\begin{align*}
V+2-\bigl(L-3(1+\delta)^2\bigr)(1+\delta)
&=
\delta\bigl(4-V_*+9\delta+3\delta^2\bigr)\\
&\quad
+(V-V_*)\bigl(1-2^{-2/3}(1+\delta)\bigr)>0.
\end{align*}
This contradicts the first case of \eqref{eq:stationary-lower}.

If \(6<V\le8\), then
\[
\bigl(L-3(1+\delta)^2\bigr)(1+\delta)
\le
\max_{s\ge0}(Ls-3s^3)
=
\frac{(V+10)^{3/2}}9
<
4\sqrt{V-2}.
\]
Here \((V+10)^3/(V-2)\) is nonincreasing on \([6,8]\), and its value at
\(6\) is \(1024<1296\). This contradicts the second case of
\eqref{eq:stationary-lower}.
\end{proof}
\begin{proof}[Proof of \Cref{cor:binding}]
By \cite[Theorem~3.2]{FrankLieb:compact-least-density}, the infimum defining \(e_*\) is attained.
An optimizer minimizes \(\cE\) for fixed volume and is therefore a ball by
\Cref{theo:main}. Since
\[
\frac{\cE(B_V)}{V}
=
(36\pi)^{1/3}
\left(
V^{-1/3}+\frac15V^{2/3}
\right),
\]
the minimum occurs at \(V=5/2\) and has the stated value.
\end{proof}

\appendix

\section{}

\begin{lemma}\label{lem:levels}
Take $u$ as in \Cref{sec:cap}. For every regular value $s\in(0,1)$, the level set $\{u=s\}$ is connected.
\end{lemma}

\begin{proof}
Extend \(u\) by one on \(K\). By the maximum principle, every component of
\(\{u>s\}\) meets \(K\). Since \(K\) is connected, \(\{u>s\}\) is connected.
Since \(u\to0\), the set \(\{u<s\}\) has a unique unbounded component, while
the minimum principle rules out bounded components. Thus
\(S^3\setminus\{u=s\}\) has exactly two components. The assertion thus follows from the Jordan Brouwer separation theorem.
\end{proof}

\begin{lemma}\label{lem:bochner}
Let \(u\) be harmonic in an open set \(U\subset\R^3\). Then the
right-hand side below is locally integrable, and
\[
\Delta|\nabla u|
\ge
|\nabla u|
\left(
|\mathrm{II}|^2+
|\nabla^\top\log|\nabla u||^2
\right)
\]
in the sense of distributions, where the right-hand side is extended by zero on \(\{|\nabla u|=0\}\).
\end{lemma}

\begin{proof}
For \(\varepsilon>0\), Bochner's formula and Cauchy--Schwarz give
\[
\Delta\bigl(|\nabla u|^2+\varepsilon\bigr)^{1/2}
=
\frac{|D^2u|^2}{\bigl(|\nabla u|^2+\varepsilon\bigr)^{1/2}}
-
\frac{|D^2u(\nabla u,\cdot)|^2}
     {\bigl(|\nabla u|^2+\varepsilon\bigr)^{3/2}}
\geq 0.
\]
Letting \(\varepsilon\downarrow0\), we find that \(|\nabla u|\) is
subharmonic. In particular, \(\Delta|\nabla u|\) is a nonnegative
Radon measure. On \(\{|\nabla u|>0\}\), choose an orthonormal frame with
\(e_3=\nabla u/|\nabla u|\) and \(e_1,e_2\) tangent to the level sets.
Writing \(u_{ij}=D^2u(e_i,e_j)\), the classical Bochner identity gives
\[
\begin{aligned}
\Delta|\nabla u|
=
\frac{|D^2u|^2-|\nabla|\nabla u||^2}{|\nabla u|} =
\frac{\displaystyle
\sum_{a,b=1}^2u_{ab}^2+\sum_{a=1}^2u_{a3}^2}
{|\nabla u|} =
|\nabla u|
\left(
|\mathrm{II}|^2+
|\nabla^\top\log|\nabla u||^2
\right).
\end{aligned}
\]
Thus the restriction of \(\Delta|\nabla u|\) to
\(\{|\nabla u|>0\}\) is given by the density on the right-hand side,
while its restriction to \(\{|\nabla u|=0\}\) is nonnegative. The
claim follows.
\end{proof}

\bibliographystyle{alpha}
\bibliography{bib}

\end{document}